\documentclass[10pt]{amsart}
\usepackage{graphicx}
\usepackage{tikz}
\usepackage{amssymb}
\usepackage{sseq}
\usepackage{dsfont}
\usepackage{mathtools}
\usepackage{enumerate}
\usepackage{doi}
\usepackage{faktor}
\usepackage[font=small,labelfont=bf]{caption}
\usepackage{mathrsfs}
\usepackage{stmaryrd}
\usepackage{verbatim}
\usepackage{manfnt}
\usepackage{array}
\usepackage{wrapfig}

\begingroup
\catcode`\&=13
\gdef\pampmatrix{%
  \begingroup
  \let&=\amsamp
  \begin{pmatrix}%
}
\gdef\endpampmatrix{\end{pmatrix}\endgroup}
\endgroup

\makeatletter							
\newtheorem*{rep@theorem}{\rep@title}
\newcommand{\newreptheorem}[2]{%
\newenvironment{rep#1}[1]{%
 \def\rep@title{#2 \ref{##1}}%
 \begin{rep@theorem}}%
 {\end{rep@theorem}}}
\makeatother

\newcommand{\im}{\textnormal{im}}

\newcommand{\F}{\mathbb{F}}
\newcommand{\Z}{\mathbb{Z}}

\newcommand{\N}{\mathbb{N}}
\newcommand{\He}{\textnormal{He}(\Z/9)}

\newcommand{\co}{\colon\thinspace}

\newcommand{\cl}{\mathsf{cl}}

\theoremstyle{plain}
\newtheorem{theorem}{Theorem}[section]
\newreptheorem{theorem}{Theorem}
\newtheorem*{theorem*}{Theorem}

\newtheorem{proposition}[theorem]{Proposition}

\newtheorem{lemma}[theorem]{Lemma}

\newtheorem{corollary}[theorem]{Corollary}

\theoremstyle{definition}
\newtheorem{definition}[theorem]{Definition}
\newtheorem{example}[theorem]{Example}

\newtheorem{question}[theorem]{Question}
\newtheorem*{question*}{Question}

\newcounter{exercises}
\theoremstyle{remark}
\newtheorem*{notation*}{Notation}
\newtheorem{remark}[theorem]{Remark}

\usepackage[all]{xy}
\SelectTips{cm}{10}
\CompileMatrices

\title{On families of nilpotent subgroups and associated coset posets}
\author{Simon Gritschacher and Bernardo Villarreal}
\address{Københavns Universitet, Institut for Matematiske Fag, Universitetsparken 5, 2100 K{\o}benhavn {\O}, Danmark}
\email{gritschacher@math.ku.dk}
\address{Universidad Nacional Aut\'{o}noma de M\'{e}xico, Instituto de Matem\'aticas, Ciudad Universitaria, Coyoac\'an 04510, CDMX, M\'{e}xico}
\email{villarreal@matem.unam.mx}
\date{\today}

\begin{document}

\maketitle

\begin{abstract}
We study some properties of the coset poset associated with the family of subgroups of class $\leq 2$ of a nilpotent group of class $\leq 3$. We prove that under certain assumptions on the group the coset poset is simply-connected if and only if the group is $2$-Engel, and $2$-connected if and only if the group is nilpotent of class $2$ or less. We determine the homotopy type of the coset poset for the group of $4\times 4$ upper unitriangular matrices over $\F_p$, and for the Burnside groups of exponent $3$.
\end{abstract}

\section{Introduction}

It is not difficult to see that a group $G$ is isomorphic to the free product of its abelian subgroups amalgamated along their intersections if and only if $G$ is abelian \cite{Ok15}. This paper is a response to \cite[Conjecture 2.1]{Ok15} which asserts that more generally a group $G$ is isomorphic to the amalgamated product of its nilpotent subgroups of class $\leq q$ if and only if $G$ is itself nilpotent of class $\leq q$. We provide counterexamples which show that this fails to hold for every $q>1$.

The problem can be embedded in a more interesting topological context using coset posets and the work of Abels and Holz \cite{AH90}. Let $G$ be a group and $\mathcal{F}$ a family of subgroups of $G$. Define $\mathscr{C}(\mathcal{F},G)$ to be the geometric realization of the poset of cosets $\{gH\in G/H\mid g\in G,\, H\in \mathcal{F}\}$ partially ordered by inclusion. Abels and Holz show in \cite[Theorem 2.4]{AH90} that the fundamental group of $\mathscr{C}(\mathcal{F},G)$ is precisely the kernel of the canonical map $\textnormal{colim}_{H\in \mathcal{F}}H\to G$, where $\mathcal{F}$ is viewed as a poset under inclusion and the colimit is that of the inclusion functor $\mathcal{F}\hookrightarrow \textnormal{Grp}$.

Therefore, letting $\mathcal{N}_{q+1}$ denote the family of nilpotent subgroups of $G$ of class $\leq q$, an equivalent formulation of \cite[Conjecture 2.1]{Ok15} asserts that $\mathscr{C}(\mathcal{N}_{q+1},G)$ is simply-connected if and only if $G$ is nilpotent of class $\leq q$. For $q=2$ it is not difficult to see that $\mathscr{C}(\mathcal{N}_{3},G)$ is always simply-connected if $G$ is a $2$-Engel group (see Proposition \ref{prop:sufficient1connected}), i.e., a group in which the identical relation $[[x,y],y]=1$ holds. Hence, any $2$-Engel group of nilpotency class $3$ is a counterexample in the case $q=2$. Examples of such groups are the Burnside groups of exponent $3$, for which we determine the homotopy type of $\mathscr{C}(\mathcal{N}_3,G)$ in Theorem \ref{thm:htB3r}.

This raises the question of how nilpotency of $G$ may be characterized in terms of higher connectivity or other combinatorial properties of $\mathscr{C}(\mathcal{N}_{q+1},G)$, which will be the focus of the second part of the paper. We expect the general situation to be difficult, which is why in this paper we focus on nilpotent groups of class $\leq 3$. One of our main results (Theorem \ref{thm:main3groups}) has the following corollary:

\begin{theorem} \label{thm:main3groupsIntro}
Let $G$ be a nilpotent group of class $\leq 3$ and suppose that $[[G,G],G]$ has exponent dividing $3$. Then $\mathscr{C}(\mathcal{N}_{3},G)$ is
\begin{itemize}
\item[(i)] 1-connected if and only if $G$ is a 2-Engel group\smallskip
\item[(ii)] 2-connected if and only if $G$ is nilpotent of class $\leq 2$.
\end{itemize}
\end{theorem}
The theorem shows that nilpotency of $G$ is linked to the topology of $\mathscr{C}(\mathcal{N}_{q+1},G)$ in a more subtle manner than proposed in \cite{Ok15}, suggesting that higher connectivity is a more reasonable hypothesis in \cite[Conjecture 2.1]{Ok15}. Note that not only simple connectivity but also higher connectivity of coset posets has a purely algebraic interpretation as detailed in \cite[\S 4]{AH90}.

While Theorem \ref{thm:main3groupsIntro} may hold more generally for any \emph{3-group} of class $\leq 3$, we do find that simple connectivity of $\mathscr{C}(\mathcal{N}_{3},G)$ does not in general imply that $G$ is a $2$-Engel group:

\begin{reptheorem}{thm:e3u42}
Let $p$ be a prime and consider the group $U_4(\F_p)$ of $4\times 4$ upper unitriangular matrices with entries in $\F_p$. Then
\begin{enumerate}
\item For $p=2$ there is a homotopy equivalence
\[
\mathscr{C}(\mathcal{N}_{3},U_4(\F_2))\simeq \bigvee^3 S^2\, .
\]
\item If $p$ is odd, then there is a homotopy equivalence
\[\mathscr{C}(\mathcal{N}_{3},U_4(\F_p))\simeq \bigvee^{p^2(p-1)^3}S^1\vee \bigvee^{(p-1)^3}S^2\,.\]
\end{enumerate}
\end{reptheorem}
By the first part of the theorem $\mathscr{C}(\mathcal{N}_{3},U_4(\F_2))$ is simply-connected, yet $U_4(\F_2)$ is not a $2$-Engel group. (Note that a $p$-group for $p\neq 3$ is $2$-Engel if and only if it is nilpotent of class $\leq 2$.) In particular, $U_4(\F_2)$ is another counterexample to the conjecture in \cite{Ok15}.

It should be mentioned that the coset posets $\mathscr{C}(\mathcal{N}_{q+1},G)$ considered here have previously featured in the work of Adem, F. Cohen and Torres-Giese \cite{ACTg12, To11} albeit in a slightly different context. A question in \cite{ACTg12} essentially asks if the complexes $\mathscr{C}(\mathcal{N}_{q+1},G)$ are aspherical. In Corollary \ref{cor:kpi1} we answer this question in the negative for $q=2$.

We end the introduction with some questions that seem reasonable from the point of view of our paper. Many variations of these questions are possible by varying degrees of nilpotency and connectivity and involving higher Engel conditions.
\begin{question} \label{qu:2engel}
Let $G$ be a $3$-group of nilpotency class $\leq 3$ and suppose that $\mathscr{C}(\mathcal{N}_{3},G)$ is simply-connected. Is $G$ necessarily a $2$-Engel group?
\end{question}

The homotopy groups $\pi_i(\mathscr{C}(\mathcal{N}_{q+1},G))$ for $i\leq k$ are obstructions for $G$ to have all $(k+1)$-generator subgroups nilpotent of class $\leq q$ (see Proposition \ref{prop:sufficient1connected}). Thus one could ask more generally:

\begin{question}
Let $G$ be a nilpotent group of class $\leq q+1$ and suppose that $\mathscr{C}(\mathcal{N}_{q+1},G)$ is $q$-connected. Is then $G$ of class $\leq q$?
\end{question}

\begin{question}
Let $G$ be a finite group and suppose that $\mathscr{C}(\mathcal{N}_{q+1},G)$ is contractible for some $q>1$. Is $G$ nilpotent?
\end{question}

\noindent \emph{Notation.} Our convention for the commutator bracket is $[x,y]=x^{-1}y^{-1}xy$ and $[x,y,z]:=[[x,y],z]$.

\section{The simplicial set $E_\ast(q+1,G)$} \label{sec:sset}

In this section we define the coset poset, and then describe a simplicial set which models up to homotopy the coset posets we are interested in. Subsequently we address \cite[Conjecture 2.1]{Ok15} as well as a question asked in \cite{ACTg12}.

\begin{definition}
Let $G$ be a group and $\mathcal{F}$ a collection of subgroups of $G$. The \emph{coset poset} associated with $\mathcal{F}$ is the set $\mathscr{C}(\mathcal{F},G):=\{gH\in G/H\mid g\in G,\, H\in\mathcal{F}\}$ partially ordered by inclusion. Its \emph{order complex} is the simplicial complex whose $n$-simplices are in bijection with the chains $gH_0\subset \cdots \subset gH_n$ in $\mathscr{C}(\mathcal{F},G)$.
\end{definition}

We will not distinguish notationally between the coset poset, the order complex and its geometric realization, and refer to them collectively as the coset poset.

To fix some of our notation we recall the following definition.
\begin{definition}
Let $G$ be a group. Define a sequence of subgroups $\Gamma^{q}(G)\leq G$ inductively as follows:
\[
\Gamma^1(G):=G\, , \quad \Gamma^{q+1}(G):=[\Gamma^q(G),G]\, \textnormal{ for all }q>0.
\]
The group $G$ is \emph{nilpotent} if there is an integer $c\geq 0$ such that $\Gamma^{c+1}(G)= 1$. The least such integer $c$ is the \emph{nilpotency class} of $G$, which we denote by $\cl(G)$.
\end{definition}

Thus, $\cl(G)=0$ if and only if $G$ is trivial, and $\cl(G)\leq 1$ if and only if $G$ is abelian.

We will be interested in $\mathscr{C}(\mathcal{N}_{q+1}(G),G)$, where $\mathcal{N}_{q+1}(G)$ (or simply $\mathcal{N}_{q+1}$ if $G$ is understood) denotes the collection of subgroups of $G$ of nilpotency class \emph{strictly less than} $q+1$. Coset posets admit several homotopy equivalent models which are described in \cite[Section I]{AH90}. For a large portion of this paper we will find it convenient to work with a simplicial set model for $\mathscr{C}(\mathcal{N}_{q+1},G)$.

\begin{definition}
We call a tuple $(x_0,\dots,x_k)\in G^{k+1}$ \emph{affinely nil-$q$} if there exist $y\in G$ and $H\leq G$ with $\cl(H)\leq q$ such that $x_i\in yH$ for all $i=0,\dots,k$.
\end{definition}

In words, a tuple $(x_0,\dots,x_k)$ is affinely nil-$q$ if all its coordinates $x_i$ are contained in a single left coset of a nilpotent group of class at most $q$ (which is sometimes called a \emph{nil-$q$ group}). Clearly if $(x_0,\dots,x_k)$ is affinely nil-$q$, then so is the tuple $(x_0,\dots,\hat{x}_i,\dots,x_k)$ with $x_i$ omitted for any $i=0,\dots,k$.

\begin{definition}
Define a simplicial set $E_\ast(q+1,G)$ by taking as $k$--simplices
\[
E_k(q+1,G):=\{(x_0,\dots,x_k)\in G^{k+1}\mid (x_0,\dots,x_k)\textnormal{ is affinely nil-$q$} \}
\]
for each $k\geq 0$. The face maps $d_i\co E_{k}(q+1,G)\to E_{k-1}(q+1,G)$ and the degeneracy maps $s_i\co E_k(q+1,G)\to E_{k+1}(q+1,G)$ are, respectively, given by
\begin{alignat*}{2}
d_i(x_0,\dots,x_k) & =(x_0,\dots,\hat{x}_i,\dots,x_k) \\
 s_i(x_0,\dots,x_k) & =(x_0,\dots,x_i,x_i,\dots,x_k)\,,
\end{alignat*}
for $0\leq i\leq k$. We denote by $E(q+1,G)$ the geometric realization of $E_\ast(q+1,G)$.
\end{definition}

\begin{proposition} \label{prop:eqgvscosetposet}
There is a homotopy equivalence $E(q+1,G)\simeq \mathscr{C}(\mathcal{N}_{q+1},G)$.
\end{proposition}
\begin{proof}
This follows as a special case of the considerations in \cite[Section~I]{AH90}. In their notation, take $\mathfrak{U}:=\{gH\mid g\in G,\, H\in \mathcal{N}_{q+1}\}$ as a cover of $G$. Then our $E(q+1,G)$ is their $|X^{\textnormal{simp}}(\mathfrak{U})|$ and our $\mathscr{C}(\mathcal{N}_{q+1},G)$ is their $|F(\mathfrak{U})|$  . Since $\mathfrak{U}$ is closed under intersections, it follows from Theorem 1.4 and 1.6 in \cite{AH90} that there is a homotopy equivalence $|X^{\textnormal{simp}}(\mathfrak{U})|\simeq |F(\mathfrak{U})|$.
\end{proof}

For the remainder of this paper we will use the models $E(q+1,G)$ and $\mathscr{C}(\mathcal{N}_{q+1},G)$ interchangeably, and we will not always be explicit about this.

\begin{remark} \label{rem:eg}
Note that $E_\ast(q+1,G)$ is a simplicial subset of $E_\ast G$, the latter being defined in the same fashion but without the condition that the simplices $(x_0,\dots,x_k)$ be affinely nil-$q$, i.e., $E_kG=G^{k+1}$. The realization $EG$ is a contractible space on which $G$ acts freely, so that $BG:=EG/G$ is a classifying space for $G$. Our notation $E(q+1,G)$ derives from \cite{ACTg12}, where the spaces $E(q+1,G)$ were studied -- somewhat independently from the literature on coset posets -- in relation to a filtration of the classifying space $BG$.
\end{remark}

\begin{remark}
It can sometimes be convenient to work with $E(q+1,G)$ instead of $\mathscr{C}(\mathcal{N}_{q+1},G)$. For example, it is immediate that for any pair of groups $G$ and $H$ there is a natural homeomorphism
\begin{equation} \label{eq:eqgproduct}
E(q+1,G\times H) \cong E(q+1,G)\times E(q+1,H)
\end{equation}
induced by the projections onto $G$ and $H$, while for the coset poset this holds only up to homotopy and is not as easy to see. Note that as a consequence of (\ref{eq:eqgproduct}), if $G$ is assumed nilpotent and finite, it is fine to assume that $G$ is a $p$-group when studying the homotopy type of $\mathscr{C}(\mathcal{N}_{q+1},G)$. Also notice that $\mathscr{C}(\mathcal{N}_{q+1},G)$ is in general not homotopy equivalent to a wedge of spheres, unlike some of the classical coset posets studied in the literature.
\end{remark}

For a subset $S\subseteq G$ let $\langle S\rangle\leqslant G$ denote the subgroup of $G$ generated by $S$. We will frequently use the following characterization of affinely nil-$q$ tuples whose proof is obvious.

\begin{lemma}
A tuple $(x_0,\dots,x_k)\in G^{k+1}$ is affinely nil-$q$ if and only if
\[
\cl( \langle x_k^{-1}x_0,\dots,x_k^{-1}x_{k-1}\rangle)\leq q\,.
\]
\end{lemma}

The next result gives a sufficient condition for $E(q+1,G)$ to be $(k-1)$--connected, i.e., for $\pi_i(E(q+1,G))=0$ for all $ i\leq k-1$.

\begin{proposition} \label{prop:sufficient1connected}
Let $G$ be a group and $k\geq 1$ an integer. Suppose that every $k$-generator subgroup of $G$ is nilpotent of class at most $q$. Then $E(q+1,G)$ is $(k-1)$--connected.
\end{proposition}
\begin{proof}
By assumption, $\cl( \langle x_k^{-1}x_0,\dots,x_k^{-1}x_{k-1}\rangle)\leq q$ for every $(x_0,\dots,x_k)\in G^{k+1}$, hence every $(k+1)$--tuple is affinely nil--$q$. Therefore, the simplicial $k$--skeleta of $E(q+1,G)$ and $EG$ agree (see Remark \ref{rem:eg}). As for every $i\leq k-1$ the $i$--th homotopy group of a simplicial set depends only on the simplicial $k$--skeleton and $EG$ is contractible, it follows that $E(q+1,G)$ is $(k-1)$--connected.
\end{proof}

\begin{definition}
A group $G$ is called a \emph{$2$-Engel group} if $[x,y,y]=1$ for all $x,y\in G$.
\end{definition}

Every $2$-Engel group is nilpotent of class $\leq 3$, and the Burnside group $B(3,3)$ (see Section \ref{sec:burnside}) is an example of a $2$-Engel group of class exactly $3$.

We are now in a position to disprove \cite[Conjecture 2.1]{Ok15}. The $2$-Engel groups can be characterized as those groups for which every two generator subgroup is nilpotent of class at most $2$ \cite[Theorem 7.15(iv)]{Ro72}. Moreover, in \cite[Corollary 1]{BM71} a nonsolvable group $G$ is constructed all of whose two generator subgroups are nilpotent of class at most $3$. Therefore, by Proposition \ref{prop:sufficient1connected}
\begin{itemize}
\item if $G$ is a $2$-Engel group, then $E(3,G)$ is simply-connected, yet $G$ need not be nilpotent of class $\leq 2$,
\item if $G$ is the group constructed in \cite{BM71}, then $E(q+1,G)$ is simply-connected for all $q\geq 3$, yet $G$ is not nilpotent.
\end{itemize}
This shows that \cite[Conjecture 2.1]{Ok15} fails to hold for every $q\geq 2$, unless further assumptions on $G$ are made.\medskip

We end this section by addressing a different question raised by Adem, F. Cohen, and Torres-Giese \cite{ACTg12}. The group $G$ acts freely on $E(q+1,G)$, so that $E(q+1,G)\to E(q+1,G)/G=:B(q+1,G)$ is a covering. In \cite[p.~100]{ACTg12} it was asked, motivated by examples found therein, if $B(q+1,G)$ is an Eilenberg-MacLane space of type $K(\pi,1)$. By \cite[Theorem 4.6]{ACTg12} this is related to the question of whether the graph of groups formed by the maximal elements in $\mathcal{N}_{q+1}$ and their pairwise intersections is in fact a \emph{tree of groups}.

For example, $B(q+1,Q_{2^r})$ is a $K(\pi,1)$ for all $r\geq 3$ and $q\geq 1$ where $Q_{2^r}$ is the generalized quaternion group of order $2^r$ \cite[\S 7]{ACV20}, while the extraspecial $p$-groups of rank $\geq 4$ were amongst the first examples of groups found for which $B(2,G)$ is not a $K(\pi,1)$ (see \cite{Ok15}). Here we provide a class of groups for which $B(3,G)$ is \emph{not} a $K(\pi,1)$, thus answering the question raised in \cite{ACTg12} for $q=2$.

\begin{corollary} \label{cor:kpi1}
Let $G$ be a $2$-Engel group with $\cl(G)=3$. Then $B(3,G)$ is not an Eilenberg-MacLane space of type $K(\pi,1)$.
\end{corollary}
\begin{proof}
As $G$ is a $2$-Engel group, $\exp(\Gamma^3(G))$ divides $3$, see \cite[Theorem 7.14]{Ro72}. By Theorem \ref{thm:main3groupsIntro}, $E(3,G)$ is simply-connected but not $2$-connected. Hence, $E(3,G)$ is not a $K(\pi,1)$. Because $E(3,G)\to B(3,G)$ is a $G$-covering, neither is $B(3,G)$.
\end{proof}

\section{Connectivity of $E(3,G)$} \label{sec:connectivity}

The main objective of this section is the proof of Theorem \ref{thm:main3groupsIntro}. In fact, we will prove the more general Theorem \ref{thm:main3groups} below, of which Theorem \ref{thm:main3groupsIntro} will be a straightforward consequence. At the end of the section, we prove another result (Theorem \ref{thm:2genp-grcl=3}) which gives some insight into the connectivity of $\mathscr{C}(\mathcal{N}_3,G)$ when $G$ is a $2$-generator $p$-group of class $\leq 3$.

In this section we assume that $G$ is nilpotent of class at most $3$. Then $\Gamma^3(G)$ is an abelian group, whose group law we will write additively. Let $e$ be the exponent of $\Gamma^3(G)$ (we set $e:=0$ if $\Gamma^3(G)$ does not have an exponent), so that $\Gamma^3(G)$ is a $\Z/e$-module. Let $C_\ast(E(3,G))$ denote the Moore complex with $\Z/e$-coefficients associated with the simplicial set $E_\ast(3,G)$, i.e.,
\[
C_k(E(3,G)):=\bigoplus_{(x_0,\dots,x_k)\in E_k(3,G)} \Z/e
\]
and the differentials $\partial\co C_k(E(3,G))\to C_{k-1}(E(3,G))$ are given by
\[
\partial(x_0,\dots,x_k)=\sum_{i=0}^k(-1)^i(x_0,\dots,\hat{x}_i,\dots,x_k)\, .
\]

We will prove Theorem \ref{thm:main3groups} by an obstruction argument using the following obstruction classes:

\begin{definition}
Define
\begin{itemize}
\item[(i)] a $\Gamma^3(G)$--valued 1-cochain $\omega_1 \co C_1(E(3,G)) \to \Gamma^3(G)$ by
\[
\omega_1(x,y) := [x,y,xy]
\]
and extending linearly,\smallskip
\item[(ii)] a $\Gamma^3(G)$--valued 2-cochain $\omega_2 \co C_2(E(3,G)) \to \Gamma^3(G)$ by
\[
\omega_2(x,y,z) := [x,y,z]
\]
and extending linearly.
\end{itemize}
\end{definition}

In the following two lemmas we will use frequently and without mention the fact that when $\cl(G)\leq 3$ the commutator bracket
\[
[-,-,-]\co G\times G\times G\to \Gamma^3(G)
\]
is a homomorphism in each argument. This follows easily from iterated application of the two identities
\begin{equation*} \label{eq:2comm}
[xy,z]=[x,z]^y[y,z]\,,\quad [x,yz]=[x,z][x,y]^z\,,
\end{equation*}
which are valid in any group \cite[5.1.5(ii)]{Ro96}, and the fact that $\Gamma^3(G)$ is central. Let $Z(G)$ be the center of $G$.

\begin{lemma} \label{lem:cocyclew1}
Suppose that for all $H\leq G$ with $\cl(H)\leq 2$ we have that $[H,H]^3\leq Z(G)$. Then $\omega_1$ is a cocycle.
\end{lemma}
\begin{proof}
Let $(x,y,z)\in E_2(3,G)$ be an affinely nil-$2$ triple. We must verify that $\omega_1(\partial(x,y,z))=0$, i.e., we must verify the cocycle condition
\[
\omega_1(y,z)-\omega_1(x,z)+\omega_1(x,y)=0\, .
\]
By definition of $\omega_1$, this is equivalent to
\[
[y,z,yz]-[x,z,xz]+[x,y,xy]=0
\]
in $\Gamma^3(G)$. First, we note that
\[
[x,y,xy]+[y,x,yx]=[x,xy,xy]+[y,xy,xy]=[xy,xy,xy]=0\, ,
\]
i.e., $\omega_1(x,y)=-\omega_1(y,x)$. Now
\[
[x,y,xy]=[zz^{-1}x,y,zz^{-1}xy]=[z,y,zy]+[z^{-1}x,y,zy]+[zz^{-1}x,y,z^{-1}x]\,,
\]
i.e.,
\[
\omega_1(x,y)+\omega_1(y,z)=[z^{-1}x,y,zy]+[zz^{-1}x,y,z^{-1}x]\, .
\]
The first summand on the right hand side equals
\begin{equation*}
\begin{split}
& [z^{-1}x,zz^{-1}y,z^{-1}yz^2]\\&\quad\quad=[z^{-1}x,z,z^{-1}y]+[z^{-1}x,z,z^2]+\underbrace{[z^{-1}x,z^{-1}y,z^{-1}y]}_{A}+[z^{-1}x,z^{-1}y,z^2]\,,
\end{split}
\end{equation*}
while the second summand equals
\begin{equation*}
\begin{split}
& [zz^{-1}x,zz^{-1}y,z^{-1}x]\\&\quad\quad=\underbrace{[z,z,z^{-1}x]}_{B}+[z,z^{-1}y,z^{-1}x]+[z^{-1}x,z,z^{-1}x]+\underbrace{[z^{-1}x,z^{-1}y,z^{-1}x]}_{C}\,.
\end{split}
\end{equation*}
The term labelled $B$ vanishes trivially, while both $A$ and $C$ vanish because $(x,y,z)$ is affinely nil-$2$ and hence $\cl(\langle z^{-1}x,z^{-1}y\rangle)\leq 2$. Combining the remaining terms we obtain
\begin{equation*}
\begin{split}
&\omega_1(x,y)+\omega_1(y,z)\\&\quad\quad=\underbrace{[z^{-1}x,z,zx]}_{=\omega_1(x,z)}+\underbrace{[z^{-1}x,z,z^{-1}y]+[z,z^{-1}y,z^{-1}x]}_{D}+[z^{-1}x,z^{-1}y,z^2]\, .
\end{split}
\end{equation*}
The term labelled $D$ can be simplified using the Hall-Witt identity. Since $\Gamma^3(G)$ is central in $G$, the Hall-Witt identity \cite[5.1.5(iv)]{Ro96} becomes
\[
[a,b,c]+[b,c,a]+[c,a,b]=0\quad \forall\,a,b,c\in G\, .
\]
Applied to $D$ this yields $D=[z^{-1}x,z^{-1}y,z]$, and combined with the last term on the right hand side we obtain
\[
\omega_1(x,y)+\omega_1(y,z)-\omega_1(x,z)=[z^{-1}x,z^{-1}y,z^3]\, .
\]
The commutator equals $[[z^{-1}x,z^{-1}y]^3,z]$, and since $\cl(\langle z^{-1}x,z^{-1}y\rangle)\leq 2$ it vanishes by the assumption on $G$. The cocycle condition now follows.
\end{proof}

\begin{lemma} \label{lem:cocyclew2}
Suppose that $G$ is a 2-Engel group. Then $\omega_2$ is a cocycle.
\end{lemma}
\begin{proof}
Let $(x,y,z,w)\in E_3(3,G)$ be an affinely nil-$2$ tuple. We must show that $\omega_2(\partial(x,y,z,w))=0$, i.e., that
\begin{equation} \label{eq:cocycle2}
[y,z,w]-[x,z,w]+[x,y,w]-[x,y,z]=0
\end{equation}
in $\Gamma^3(G)$. First, we note that
\begin{equation*}
\begin{split}
& [x^{-1}y,x^{-1}z,x^{-1}w]\\&\quad=\underbrace{[x^{-1},z,x^{-1}]}_{A}+[x^{-1},z,w]+\underbrace{[y,x^{-1},x^{-1}]}_{B}+[y,x^{-1},w]+[y,z,x^{-1}]+[y,z,w]\,.
\end{split}
\end{equation*}
The terms labelled $A$ and $B$ vanish because $G$ is a $2$-Engel group. Moreover, since $(x,y,z,w)$ is affinely nil-$2$, the left hand side vanishes. Rewriting the remaining terms on the right hand side we obtain
\[
0=-[x,z,w]+[x,y,w]-\underbrace{[y,z,x]}_{C}+[y,z,w]\, .
\]
In every $2$-Engel group $G$ the identity
\[
[a,b,c]=[c,a,b]\quad \forall\, a,b,c\in G
\]
holds \cite[12.3.6(ii)]{Ro96}. Therefore, $C=[x,y,z]$ and the identity (\ref{eq:cocycle2}) follows.
\end{proof}

We can now prove the following more general form of Theorem \ref{thm:main3groupsIntro}.

\begin{theorem} \label{thm:main3groups}
Let $G$ be a nilpotent group of class $\cl(G)\leq 3$ and let $e\in \N\cup \{0\}$ be the exponent of $\Gamma^3(G)$. Suppose that for every $H\leq G$ with $\cl(H)\leq 2$ we have that $[H,H]^3\leq Z(G)$. Then
\begin{itemize}
\item[(i)] $H_1(E(3,G);\Z/e)=0$ if and only if $G$ is a 2-Engel group\smallskip
\item[(ii)] $H_i(E(3,G);\Z/e)=0$ for $i=1,2$ if and only if $\cl(G)\leq 2$.
\end{itemize}
\end{theorem}
\begin{proof}
(i)  By the assumptions of the theorem and by Lemma \ref{lem:cocyclew1} a cohomology class $[\omega_1]\in H^1(E(3,G);\Gamma^3(G))$ is defined. For every $x,y\in G$ the chain
\[
c(x,y):=(x,y)+(y,1)+(1,x)\in C_1(E(3,G))
\]
is a homology cycle. Now $[\omega_1]$ defines a homomorphism
\[
-\cap \omega_1 \co H_1(E(3,G);\Z/e)\to \Gamma^3(G)
\]
which maps $[c(x,y)]$ to $[x,y,xy]$. If $H_1(E(3,G);\Z/e)=0$, then $-\cap \omega_1=0$, hence
\[
0= c(x,y) \cap \omega_1=[x,y,xy]
\]
for all $x,y\in G$. Now $[x,y,xy]=[x,xy,xy]$ and replacing $y$ by $x^{-1}y'$ we find that $[x,y',y']=0$ for all $x,y'\in G$. Hence, $G$ is a 2-Engel group. Conversely, if $G$ is a $2$-Engel group, then $E(3,G)$ is simply-connected by Proposition \ref{prop:sufficient1connected}. In particular, $H_1(E(3,G);\Z/e)=0$.

(ii) Suppose that $H_i(E(3,G);\Z/e)=0$ for $i=1,2$. By (i), $G$ is a $2$-Engel group, hence Lemma \ref{lem:cocyclew2} shows that $[\omega_2]\in H^2(E(3,G);\Gamma^3(G))$ is defined. Since $G$ is a 2-Engel group, every triple $(x,y,z)\in G^3$ is affinely nil-$2$, which shows that for all $x,y,z\in G$ there exists a homology cycle
\[
c(x,y,z):=(x,y,z)-(1,y,z)+(1,x,z)-(1,x,y) \in C_2(E(3,G))\, .
\]
Since $H_2(E(3,G);\Z/e)=0$, the homomorphism $-\cap \omega_2\co H_2(E(3,G);\Z/e)\to \Gamma^3(G)$ is trivial, hence
\[
0=c(x,y,z)\cap \omega_2=[x,y,z]
\]
for all $x,y,z\in G$, i.e., $\cl(G)\leq 2$. Conversely, if $\cl(G)\leq 2$, then $E(3,G)=EG$ which is contractible.
\end{proof}

\begin{remark} \label{rem:centrality}
It would be interesting to characterize the groups which satisfy the condition of the theorem. Note that if $\cl(G)\leq 3$ and $x_1,y_1,\dots,x_k,y_k\in G$ are such that $[x_i,y_i]^3\in Z(G)$ for all $i=1,\dots,k$, then $\prod_{i=1}^k [x_i,y_i]^3\in Z(G)$. Because of this and because $[G,G]$ is abelian, the centrality condition in the theorem is equivalent to the condition that $[x,y]^3\in Z(G)$ for all $x,y\in G$ with $\cl(\langle x,y\rangle)\leq 2$.
\end{remark}

\begin{proof}[Proof of Theorem \ref{thm:main3groupsIntro}]
Let $G$ be a group with $\cl(G)\leq 3$ and $\exp(\Gamma^3(G))\mid 3$. Then for every $x,y,z\in G$ we have $1=[x,y,z]^3=[[x,y]^3,z]$, hence $[x,y]^3\in Z(G)$. Therefore, $G$ satisfies the assumption of Theorem \ref{thm:main3groups} (see Remark \ref{rem:centrality}). Moreover, the proof of Theorem \ref{thm:main3groups} shows that $H_1(E(3,G);\Z/e)=0$ if and only if $\pi_1(E(3,G))=0$, and $H_i(E(3,G);\Z/e)=0$ for $i=1,2$ if and only if $E(3,G)$ is $2$-connected. Thus, Theorem \ref{thm:main3groupsIntro} is a consequence of Theorem \ref{thm:main3groups}, noting that $E(3,G)\simeq \mathscr{C}(\mathcal{N}_3,G)$ by Proposition \ref{prop:eqgvscosetposet}.
\end{proof}

However, there are groups $G$ to which Theorem \ref{thm:main3groups} applies and which have $\exp(\Gamma^3(G))>3$. We demonstrate this in the next two examples.

\begin{example} Let $\He$ denote the Heisenberg group over $\Z/9$. Then any central extension of $\He$ satisfies the assumptions of Theorem \ref{thm:main3groups}.
\end{example}
\begin{proof}
Let $K \xrightarrow{i} G \xrightarrow{p} \He$ be a central extension. We may assume that $\cl(G)=3$. Let $x,y\in G$ and suppose that $\cl(\langle x,y\rangle)\leq 2$. We will show that $[x,y]^3\in Z(G)$ by showing that $[p(x),p(y)]^3=1$, and therefore $[x,y]^3\in \im(i)\leq Z(G)$. It suffices to show that $[p(x),p(y)]$ does not generate the center $Z(\He)\cong C_9$, for then $[p(x),p(y)]=c^3$ for some $c\in Z(\He)$, hence $[p(x),p(y)]^3=1$.

Now suppose for contradiction that $[p(x),p(y)]$ generates $Z(\He)$. Since $\He$ is 2-generated, the Frattini quotient $\He/\Phi(\He)$ is elementary abelian of rank $2$. The images of $p(x)$ and $p(y)$ in $\He/\Phi(\He)$ must be linearly independent. If this were not the case, then $p(y)=p(x)^d$ mod $\Phi(\He)$ for some $d\in \Z$, hence $[p(x),p(y)]\in \He^3$, and $[p(x),p(y)]$ would not generate $Z(\He)$. Therefore, $p(x)$ and $p(y)$ map to a basis of $\He/\Phi(\He)$, whence $\langle p(x),p(y)\rangle=\He$. But then $G=p^{-1}(\langle p(x),p(y)\rangle)=\langle x,y\rangle K$, so $\cl(\langle x,y\rangle)=\cl(G)=3$, contradicting the assumption that $\cl(\langle x,y\rangle)\leq 2$.
\end{proof}

\begin{example} \label{ex:centralextensionHe9}
The central extension of $\He$ by the cyclic group $C_9$ with presentation
\[
G=\langle x,y,z\mid x^9,\, y^9,\, [x,z],\, [y,z],\, [x,y,y]z^{-1},\, [y,x,x]z^{-1}\rangle
\]
has $\exp(\Gamma^3(G))=9$.
\end{example}

Recall that \cite[Conjecture 2.1]{Ok15} asserted that $\pi_1(\mathscr{C}(\mathcal{N}_3,G))=0$ if and only if $\cl(G)\leq 2$. While not true in general, this assertion holds when $G$ is a two generator $p$-group of class at most $3$. Proving this will be our goal for the remainder of this section (see Theorem \ref{thm:2genp-grcl=3}).

If $N\vartriangleleft G$ is a normal subgroup and $g\in G$ and $H\in \mathcal{F}$ we write $\bar{g}$ and $\bar{H}$ for the image of $g$ respectively $H$ under the quotient map $G\to G/N$. We also write $\bar{\mathcal{F}}$ for the collection $\{\bar{H}\mid H\in \mathcal{F}\}$ of subgroups of $G/N$.

\begin{lemma} \label{lem:quotientcosetposet}
Let $N\vartriangleleft G$ be a normal subgroup and suppose that for every $H\leq G$ with $\cl(H)\leq 2$ we have that $\cl(NH)\leq 2$. Then the map
\[
p\co \mathscr{C}(\mathcal{N}_{3},G )\to \mathscr{C}(\bar{\mathcal{N}}_{3},G/N)
\]
sending $gH\mapsto \bar{g}\bar{H}$ is a homotopy equivalence.
\end{lemma}
\begin{proof}
This follows from Quillen's fiber lemma \cite[Proposition 1.6]{Quillen78}, because for each $\bar{g}\bar{H}\in \mathscr{C}(\bar{\mathcal{N}}_{3},G/N)$ the fiber $p/\bar{g}\bar{H}:=\{g'H'\in \mathscr{C}(\mathcal{N}_{3},G)\mid p(g'H')\leq \bar{g}\bar{H}\}$ has $gNH$ as maximal element, and thus is contractible.
\end{proof}

\begin{theorem}\label{thm:2genp-grcl=3}
Let $G$ be a nilpotent group of class $\cl(G)\leq 3$ and let $e\in \N\cup \{0\}$ be the exponent of $\Gamma^3(G)$. Suppose that either
\begin{enumerate}
\item $G$ is a two generator $p$-group for some prime $p$, or
\item for every subgroup $H\leq G$ with $\cl(H)\leq 2$ we have that $[H,H]\leq Z(G)$.
\end{enumerate}
Then $H_1(E(3,G);\Z/e)=0$ if and only if $\cl(G)\leq 2$.
\end{theorem}
\begin{proof}
Clearly, if $\cl(G)\leq 2$, then $H_1(E(3,G);\Z/e)=0$ because $E(3,G)$ is contractible. To prove the converse assume that $H_1(E(3,G);\Z/e)=0$.

Let us first consider the case when $G$ is a two generator $p$-group. Assume for contradiction that $\cl(G)=3$ and let $p^k$, where $k>0$, be the exponent of $\Gamma^3(G)$. Define $\bar{G}:=G/\Gamma^3(G)^p$ and consider the map $f\co E(3,G) \to E(3,\bar{G})$ induced by the quotient map $G\to \bar{G}$. Then $f$ is surjective on fundamental groups. To see this note that the simplicial $1$-skeleta of $E(3,G)$ and $E(3,\bar{G})$ are the complete graphs on the sets $G$ respectively $\bar{G}$, and $f$ is surjective on vertices as well as edges. It follows now from the Hurewicz theorem that $H_1(E(3,\bar{G});\Z/e)=0$. 

Now $\bar{G}$ is two generated of class three with $\exp(\Gamma^3(\bar{G}))=p$. Thus, according to \cite[~Theorem~3.7]{LQZ17} every proper subgroup of $\bar{G}$ is of class at most two. Next we apply Lemma \ref{lem:quotientcosetposet} to $\bar{G}$ with $N=\bar{G}^p[\bar{G},\bar{G}]$, the Frattini subgroup of $\bar{G}$. We have that $\bar{G}/N\cong (\Z/p)^2$ and that $\bar{\mathcal{N}}_{3}$ agrees with the poset of all proper subgroups of $\bar{G}/N$. Thus $\mathscr{C}(\mathcal{N}_3,\bar{G})$ is homotopy equivalent to the realization of the coset poset of all proper subspaces of $(\Z/p)^2$. This is a $1$-dimensional complex, and a brief computation shows that its Euler characteristic equals $(p^2-1)(p-1)+1$. In particular, $H_1(\mathscr{C}(\mathcal{N}_3,\bar{G});\Z/e)\neq 0$, and since $E(3,\bar{G})\simeq \mathscr{C}(\mathcal{N}_3,\bar{G})$ we have a contradiction.

Finally, assume that $G$ satisfies (2). By Theorem \ref{thm:main3groups}, $G$ is a $2$-Engel group. Thus for all $x,y\in G$ we have that $\cl(\langle x,y\rangle)\leq 2$ and hence $[x,y]\in Z(G)$ by (2). It follows that $\cl(G)\leq 2$.
\end{proof}

\begin{example}
Let $U_4(\F_2)\leq GL_4(\F_2)$ denote the subgroup of upper triangular matrices with ones on the diagonal. Then $\cl(U_4(\F_2))=3$, but $E(3,U_4(\mathbb{F}_2))$ is simply-connected (see Theorem \ref{thm:e3u42}). However, this does not conflict with Theorem \ref{thm:2genp-grcl=3}, because $U_4(\F_2)$ is minimally generated by three elements, and it contains a subgroup isomorphic to the Heisenberg group over $\F_2$ whose derived subgroup is not in the center of $U_4(\F_2)$.
\end{example}

\begin{remark}
Similarly to Theorem \ref{thm:main3groups}, the condition $H_1(E(3,G);\Z/e)=0$ in Theorem \ref{thm:2genp-grcl=3} is equivalent to the \emph{a priori} stronger condition that $\pi_1(E(3,G))=0$.
\end{remark}

\begin{remark}
Let $G$ be a $p$-group all of whose proper subgroups are nilpotent of class at most $2$. Then $G$ is of class at most $3$ (see \cite[\S 4 Corollary 1]{Macd1}), and $E(3,G)$ is homotopy equivalent to the coset poset of all proper subgroups of $G$. By a theorem of Brown \cite[Proposition 11]{Br00}, $E(3,G)$ has the homotopy type of a wedge of spheres. The dimension of the spheres could be computed from a chief series for $G$, but if $G$ is two generated, then Theorem \ref{thm:2genp-grcl=3} readily implies that $E(3,G)$ has the homotopy type of a bouquet of circles. In particular, $B(3,G)$ is a $K(\pi,1)$ in this situation (cf. Corollary \ref{cor:kpi1}).
\end{remark}

\section{Homotopy types} \label{sec:homotopytypes}

A natural counterexample to \cite[Conjecture 2.1]{Ok15} is a $2$-Engel group of nilpotency class $3$, like the Burnside group $B(3,3)$ of rank and exponent  $3$. One of the smallest counterexamples is the group $U_4(\F_2)$ of upper unitriangular matrices over $\F_2$, which is not a $2$-Engel group. In this section we take these two examples as a motivation to determine the full homotopy type of the coset poset $\mathscr{C}(\mathcal{N}_3,G)$ when $G$ is the upper unitriangular group $U_4(\F_p)$ for any prime $p$, or the Burnside group $B(r,3)$ of exponent $3$ on any number of $r\geq 3$ generators.

\subsection{Upper unitriangular matrices} \label{sec:unipotent}

Let $U_4(\F_p)\leq GL_4(\F_p)$ be the subgroup of upper triangular matrices with ones on the diagonal. This is a nilpotent group of class $3$ (see \cite[p.~127]{Ro96}).

\begin{theorem} \label{thm:e3u42}
Let $p$ be a prime.
\begin{enumerate}
\item If $p=2$, then there is a homotopy equivalence
\[
\mathscr{C}(\mathcal{N}_3,U_4(\F_p))\simeq \bigvee^3 S^2\, .
\]
\item If $p>2$, then there is a homotopy equivalence
\[
\mathscr{C}(\mathcal{N}_3,U_4(\F_p))\simeq \bigvee^{p^2(p-1)^3}S^1\vee \bigvee^{(p-1)^3}S^2\,.
\]
\end{enumerate}
\end{theorem}

The rest of this subsection is devoted to the proof of Theorem \ref{thm:e3u42}. Our strategy will be to use Lemma \ref{lem:quotientcosetposet} to reduce the study of the homotopy type to a coset poset which is easier to analyze.

To keep the notation simple let us denote $G=U_4(\F_p)$. The group $G$ is generated by $a=I+E_{12}$, $b=I+E_{23}$, and $c=I+E_{34}$, where $E_{ij}$ is the elementary matrix with $1$ in the $(i,j)$-th spot and zeroes everywhere else. The commutator group $[G,G]$ is the subgroup of all those matrices with vanishing first super diagonal. In the first super diagonal multiplication in $G$ corresponds to addition. Hence,
\[
V:=G/[G,G]\cong (\F_p)^3\,,
\] 
and the quotient map $\pi\co G\to V$ sends the generators $\{a,b,c\}$ to the standard basis of $(\F_p)^3$.
Since $\cl(G) = 3$, we can apply Lemma \ref{lem:quotientcosetposet} with $N=[G,G]$. This yields a homotopy equivalence
\[
\mathscr{C}(\mathcal{N}_3,G)\xrightarrow{\simeq} \mathscr{C}(\mathcal{I},V)\,,
\]
where $\mathcal{I}:=\{U\leq V\mid \cl(\pi^{-1}(U))\leq 2\}$. To prove the theorem we will analyse the homotopy type of $\mathscr{C}(\mathcal{I},V)$.

In a first step we describe the family $\mathcal{I}$. The group $\Gamma^3(G)$ is isomorphic to $\F_p$ and generated by $I+E_{14}$. The $3$-fold commutator $[-,-,-]\co G^3\to \Gamma^3(G)\cong \F_p$ factors through a trilinear form $\beta\colon V^3\to \F_p$. Clearly,
\[
U \in \mathcal{I}\; \Longleftrightarrow\; \beta|_{U^3}\equiv 0\,.
\]

To analyze $\mathcal{I}$ further we derive an explicit formula for $\beta$. For $x\in G$ we write $\bar{x}:=\pi(x)\in V$. Pick the standard basis $\{e_1, e_2, e_3\}$ for $V$. If $x$ is displayed as a matrix with coordinates $(x_{ij})$, then $\bar{x}$ is the column vector $(x_{12},x_{23},x_{34})^T\in V$ with respect to the standard basis. Let $\varphi \co V\to V$ be the linear map defined by $\varphi(e_1)=-e_1$, $\varphi(e_2)=0$ and $\varphi(e_3)=e_3$. 

\begin{lemma}\label{lem:commutatorformula}
For $\bar{x},\bar{y},\bar{z}\in V$ we have $\beta(\bar{x},\bar{y},\bar{z}) = \det(\bar{x}\mid \bar{y}\mid \varphi(\bar{z}))$.
\end{lemma}
\begin{proof}
Given $x,y,z\in U_4(\F_p)$ one can verify that
\[
[x,y,z]=y_{23}(x_{12}z_{34}+x_{34}z_{12}) - x_{23}(y_{12}z_{34}+y_{34}z_{12})=\det(\bar{x}\mid \bar{y}\mid \varphi(\bar{z}))\,. \qedhere
\]
\end{proof}

We conclude that $U\in \mathcal{I}$ if and only if for all $\bar{x},\bar{y},\bar{z}\in U$ the vectors $\bar{x},\bar{y},\varphi(\bar{z})$ are linearly dependent.

\begin{lemma}\label{lem:isotrsubV}
$\mathcal{I}$ contains all subspaces of $V$ of dimension $\leq 1$. Moreover,
\begin{enumerate}
\item if $p=2$, then the $2$-dimensional subspaces in $\mathcal{I}$ are
\[
\langle e_1,e_2\rangle\,,\; \langle e_1,e_3\rangle\,,\; \langle e_2,e_3\rangle\,,\; \langle e_1+e_3,e_2\rangle\,,
\]
\item if $p$ is odd, then the $2$-dimensional subspaces in $\mathcal{I}$ are
\[
\langle e_1,e_2\rangle\,,\; \langle e_1,e_3\rangle\,,\; \langle e_2,e_3\rangle\,.
\]
\end{enumerate}
\end{lemma}
\begin{proof}
It is clear that $\mathcal{I}$ contains all subspaces of dimension $\leq 1$. Let $U\leq V$ with $\dim(U)=2$. Using Lemma \ref{lem:commutatorformula} basic linear algebra shows that $U\in \mathcal{I}$ if and only if $\varphi(U)\subseteq U$. The lists of subspaces in (1) and (2) are now easily obtained.
\end{proof}

In the next lemma we will use a result from \cite{BB} to break $\mathscr{C}(\mathcal{I},V)$ down into simpler pieces. To recall the setup let $G$ be a group and $\mathcal{F}$ a family of proper subgroups of $G$. If $N\vartriangleleft G$ is a normal subgroup, then $\mathcal{F}$ can be written as a disjoint union $\mathcal{F}=\mathcal{F}_N\sqcup \mathcal{F}^N$, where
\[
\mathcal{F}_N=\{H\in \mathcal{F}\mid HN\neq G\}\quad \textnormal{and}\quad \mathcal{F}^N=\{H\in \mathcal{F}\mid HN=G\}\,.
\]
We say that $N$ \emph{divides} $\mathcal{F}$ if for all $H\in \mathcal{F}_N$ one has $HN\in \mathcal{F}$, and for all $H\in \mathcal{F}_N$ and $K\in \mathcal{F}^N$ one has $HN\cap K\in \mathcal{F}$. If $N$ divides $\mathcal{F}$ and $\mathcal{F}^N$ is non-empty, then according to \cite[Proposition 3.24]{BB} there is a homotopy equivalence
\begin{equation} \label{eq:benj}
\mathscr{C}(\mathcal{F},G)\simeq \mathscr{C}(\mathcal{F}^N,G) \ast \mathscr{C}(\overline{\mathcal{F}}_N,G/N)\, ,
\end{equation}
where $\overline{\mathcal{F}}_N$ denotes the image of the family $\mathcal{F}_N$ in $G/N$.

\begin{lemma}\label{lem:isotr&compl}
Let $W\leq V$ be a any fixed $2$-dimensional subspace in $\mathcal{I}$. Then for any fixed choice of basepoint for $\mathscr{C}(\mathcal{I}^W,V)$ there is an unbased homotopy equivalence
\[
\mathscr{C}(\mathcal{I},V) \simeq \bigvee^{p-1}\Sigma\mathscr{C}(\mathcal{I}^W,V)\,.
\]
\end{lemma}
\begin{proof}
It is readily checked that $W$ divides $\mathcal{I}$, thus by (\ref{eq:benj})
\[
\mathscr{C}(\mathcal{I},V) \simeq \mathscr{C}(\mathcal{I}^W,V)\ast  \mathscr{C}(\overline{\mathcal{I}}_W, V/W)\, .
\]
Since $\overline{\mathcal{I}}_W=\{0\}$ consists only of the trivial subspace of $V/W\cong \F_p$, the coset poset $\mathscr{C}(\overline{\mathcal{I}}_W, V/W)$ is a discrete space with $p$ points, i.e., homeomorphic to $\bigvee^{p-1}S^0$. Fix any basepoint in $\mathscr{C}(\mathcal{I}^W,V)$. Then there are unbased equivalences
\[
\mathscr{C}(\mathcal{I},V)\simeq \mathscr{C}(\mathcal{I}^W,V) \ast \bigvee^{p-1}S^0 \simeq \Sigma \left(\mathscr{C}(\mathcal{I}^W,V) \wedge \bigvee^{p-1}S^0\right)  \cong \bigvee^{p-1}\Sigma\mathscr{C}(\mathcal{I}^W,V) \,,
\]
where we used the homotopy equivalence $X\ast Y\simeq \Sigma (X\wedge Y)$ for any pair of pointed CW-complexes $X,Y$.
\end{proof}

Now we can prove the main result of this subsection.

\begin{proof}[Proof of Theorem \ref{thm:e3u42}]
We choose $W:=\langle e_1,e_2\rangle$ in Lemma \ref{lem:isotr&compl}. Then we only need determine  the homotopy type of $\mathscr{C}(\mathcal{I}^W,V)$. The collection $\mathcal{I}^W$ consists of all the $1$-dimensional subspaces of $V$ with non-zero $e_3$-coordinate as well as all $U\in \mathcal{I}$ with $\dim(U)=2$ and $U\neq W$. The latter are listed in Lemma \ref{lem:isotrsubV}. In particular, $\dim(\mathscr{C}(\mathcal{I}^W,V))\leq 1$.

(1) Let $p=2$. We claim that $\mathscr{C}(\mathcal{I}^W,V)$ is connected. Note that the collection $\mathcal{I}^W$, partially ordered by inclusion, forms a subposet of $\mathscr{C}(\mathcal{I}^W,V)$ which is easily seen to be connected. Therefore, for every $v\in V$ the subposet $v+\mathcal{I}^W$ is connected as well. It now suffices to show that for every $v\in V$ there is a zig-zag in $\mathscr{C}(\mathcal{I}^W,V)$ between $v+\langle e_3\rangle\in v+\mathcal{I}^W$ and $\langle e_1,e_3\rangle\in \mathcal{I}^W$. Suppose that $v=\alpha e_1+\beta e_2+\gamma e_3$ for $\alpha,\beta,\gamma\in\F_2$. Then
\[
v+\langle e_3\rangle\leq \alpha e_1 +\langle e_2,e_3\rangle \geq \alpha e_1 + \langle e_3\rangle\leq \langle e_1,e_3\rangle
\]
is such a zig-zag. It follows that $\mathscr{C}(\mathcal{I}^W,V)$ is connected, hence it has the homotopy type of a wedge of circles. To obtain the number of circles we compute the Euler characteristic. For this recall that in a coset poset $\mathscr{C}(\mathcal{F},G)$ we find exactly $|G:H_0|$ different chains (or $k$-simplices) of the form $gH_0\subset \cdots \subset gH_k$ for any fixed choice of subgroups $H_0 < \cdots < H_k$ in $\mathcal{F}$. Now it is an easy calculation that $\chi(\mathscr{C}(\mathcal{I}^W,V))=-2$, hence $\mathscr{C}(\mathcal{I}^W,V)$ is a bouquet of $3$ circles. 

(2) Let $p$ be odd. Amongst the $p^2$ $1$-dimensional subspaces in $\mathcal{I}^W$ there are $(p-1)^2$ which are neither contained in $\langle e_1,e_3\rangle$ nor in $\langle e_2,e_3\rangle$. When we pass to the coset poset they give rise to $|\F_p^3:\F_p|(p-1)^2=p^2(p-1)^2$ isolated points in $\mathscr{C}(\mathcal{I}^W,V)$. The remaining $(2p-1)$ $1$-dimensional subspaces together with $\langle e_1,e_3\rangle$ and $\langle e_2,e_3\rangle$ form a connected subposet $\mathcal{I}^W_0$ of $\mathscr{C}(\mathcal{I}^W,V)$. In particular, $\mathcal{I}^W_0$ contains $\langle e_3\rangle$, and the same argument as in (1) shows that $\mathcal{I}^W_0$ and $v+\mathcal{I}^W_0$ lie in the same connected component of $\mathscr{C}(\mathcal{I}^W,V)$ for every $v\in V$. Let $\mathscr{C}_0$ denote this connected component. Then
\[
\mathscr{C}(\mathcal{I}^W,V)\simeq \bigvee^{p^2(p-1)^2}S^0 \vee \bigvee^{1-\chi(\mathscr{C}_0)} S^1\, .
\]
The proof is finished by calculating $1-\chi(\mathscr{C}_0)=(p-1)^2$ and applying Lemma \ref{lem:isotr&compl}.
\end{proof}

\subsection{The Burnside groups $B(r,3)$} \label{sec:burnside}

Recall that the Burnside group of exponent $3$ and rank $r$ is defined by
\[
B(r,3):=F_r/(F_r)^3\,,
\]
where $F_r$ is the free group on $r$ generators. It is well known that $B(r,3)$ is a finite $2$-Engel group of class $3$ if $r\geq 3$ \cite[14.2.3]{Ro96}. Let us denote the Frattini quotient by
\[
V:=B(r,3)/\Phi(B(r,3))\cong (\F_3)^r\,,
\]
and the quotient map by $\pi\co B(r,3)\to V$.

\begin{lemma}\label{lemm:isotrB3r}
Let $r\geq 3$ and let $\mathcal{I}$ be the collection of subspaces $U\leq V$ for which $\cl(\pi^{-1}(U))\leq 2$. Then $U\in \mathcal{I}$ if and only if $\dim(U)\leq 2$.
\end{lemma}
\begin{proof}
Any subspace of $V$ of dimension $\leq 2$ is in $\mathcal{I}$, because $B(r,3)$ is a $2$-Engel group. We claim that any $3$-dimensional subspace  $U\leq V$ lifts to a subgroup of $B(r,3)$ isomorphic to $B(3,3)$ whence $\cl(\pi^{-1}(U))=3$. Let $\{x_1,\dots,x_r\}\subset F_r$ be a generating set and $\{\bar{x}_1,\dots,\bar{x}_r\}$ its image in $B(r,3)$. If $\{x_1',\dots,x_r'\}$ is any other generating set of $B(r,3)$, there is an automorphism of $B(r,3)$ taking $x'_i\mapsto \bar{x}_i$ for all $i=1,\dots,r$. Now given a basis $u_1,u_2,u_3$ of $U$, extend it to a basis $u_1,\dots,u_r$ of $V$ and choose lifts $\tilde{u}_1,\dots,\tilde{u}_r$ in $B(r,3)$. Then $\tilde{u}_1,\dots,\tilde{u}_r$ is a generating set for $B(r,3)$ which can be rearranged by an automorphism so that $\tilde{u}_i \mapsto \bar{x}_i$ for all $i$, and clearly $\langle \bar{x}_1,\bar{x}_2,\bar{x}_3\rangle$ is a subgroup of $B(r,3)$ isomorphic to $B(3,3)$.
\end{proof}

\begin{theorem}\label{thm:htB3r}
For each $r\geq 3$ there is a homotopy equivalence
\[
\mathscr{C}(\mathcal{N}_3,B(r,3))\simeq \bigvee^{m(r)} S^2\,,
\]
where $m(r)=3^{r-3}(9^r-13\cdot 3^r+39)-1$.
\end{theorem}

\begin{proof}
By Lemma \ref{lem:quotientcosetposet} applied with $N=\Phi(B(r,3))=[B(r,3),B(r,3)]$ we have a homotopy equivalence
\[
\mathscr{C}(\mathcal{N}_3,B(r,3))\xrightarrow{\simeq} \mathscr{C}(\mathcal{I},V)\,.
\]
By Lemma \ref{lemm:isotrB3r}, $\mathscr{C}(\mathcal{I},V)$ is a $2$-dimensional simply-connected complex, thus it is homotopy equivalent to a wedge of $2$-spheres. The number of spheres is determined by the Euler characteristic: $m(r)=-1+n_0-n_1+n_2$, where the number $n_k$ of $k$-simplices of $\mathscr{C}(\mathcal{I},V)$ is readily computed as in part (1) of the proof of Theorem \ref{thm:e3u42}:
\begin{alignat*}{1}
n_0 & =3^r+3^{r-1}\binom{r}{1}_3+3^{r-2}\binom{r}{2}_3 \\
n_1 & = 3^{r}\binom{r}{1}_3+3^{r}\binom{r}{2}_3+3^{r-1}\binom{2}{1}_3 \binom{r}{2}_3 \\
n_2 & = 3^{r}\binom{2}{1}_3\binom{r}{2}_3\,,
\end{alignat*}
where $\binom{r}{k}_3$ is the number of $k$-dimensional subspaces of $(\F_3)^r$. A short computation yields the claimed formula for $m(r)$.
\end{proof}

\section*{Acknowledgements}


BV acknowledges support from the European Research Council (ERC) under the European Union's Horizon 2020 research and innovation programme (grant agreement No. 682922), as well as support from Universidad Nacional Aut\'onoma de M\'exico (UNAM) under the programme ``Becas de Posdoc DGAPA."

This project was also supported by the Danish National Research Foundation through the Copenhagen Centre for Geometry and Topology (GEOTOP-DNRF151).

\bibliographystyle{plain}

\end{document}